\documentclass[11pt,a4paper,reqno]{amsart} 
\usepackage{amssymb}
\usepackage{latexsym}
\usepackage{amsmath}
\usepackage{mathrsfs}
\usepackage{color}
\usepackage{calc}                         
\usepackage{todonotes} 

\newcommand{\scal}[2]{\langle #1,#2\rangle}

\newcommand{\rd}{\mathbf R^d}
\newcommand{\rr}[1]{\mathbf R^{#1}}
\newcommand{\nn}[1]{\mathbf N^{#1}}

\newcommand{\cc}[1]{\mathbf C^{#1}}

\newcommand{\nm}[2]{\Vert #1\Vert _{#2}}

\newcommand{\cdo}{\, \cdot \, }

\newcommand{\norm}[1]{\Vert#1\Vert }

\newcommand{\vrum}{\vspace{0.1cm}}

\newcommand{\maclH}{\mathcal H}
\newcommand{\maclS}{\mathcal S}
\newcommand{\mascP}{\mathscr P}
\newcommand{\mascS}{\mathscr S}

\numberwithin{equation}{section}          

\newtheorem{thm}{Theorem}
\numberwithin{thm}{section}
\newtheorem{prop}[thm]{Proposition}

\newtheorem{lemma}[thm]{Lemma}

\theoremstyle{definition}

\theoremstyle{remark}

\newtheorem{rem}[thm]{Remark}              

\author{Ahmed Abdeljawad}

\address{Johann Radon Institute for Computational
and Applied Mathematics,
Austrian Academy of Sciences,
Linz, Austria.}

\email{ahmed.abdeljawad@ricam.oeaw.ac.at}

\author{Carmen Fern\'andez}

\address{Departament d' An\`{a}lisi Matem\`{a}tica, Universitat
de Val\`{e}ncia, Valencia, Spain}

\email{fernand@uv.es}

\author{Antonio Galbis}

\address{Departament d' An\`{a}lisi Matem\`{a}tica, Universitat
de Val\`{e}ncia, Valencia, Spain}

\email{antonio.galbis@uv.es}

\author{Joachim Toft}

\address{Department of Mathematics,
Linn{\ae}us University, V{\"a}xj{\"o}, Sweden}

\email{joachim.toft@lnu.se}

\author{R{\"u}ya {\"U}ster}

\address{Department of Mathematics, Istanbul University, Istanbul, Turkey}

\email{ruya.uster@istanbul.edu.tr}

\thanks{C. Fern\'andez and A. Galbis were partially supported
by the projects  MTM2016-76647-P, ACOMP/2015/186 (Spain).}

\title{Characterizations of a class of Pilipovi{\'c} spaces by
powers of harmonic oscillator}

\begin{document}

\begin{abstract}
We show that a smooth function $f$ on $\rr d$ belongs to
the Pilipovi{\'c} space
$\maclH _{\flat _\sigma}(\rr d)$ or the Pilipovi{\'c} space
$\maclH _{0,\flat _\sigma}(\rr d)$, if and only if the $L^p$ norm of
$H_d^Nf$ for $N\ge 0$, satisfy certain types of estimates. Here
$H_d=|x|^2-\Delta _x$ is the harmonic oscillator.
\end{abstract}

\keywords{Harmonic oscillator, Pilipovi{\'c} spaces}

\subjclass{46F05, 42B35, 30Gxx, 44A15}

\maketitle

\section{Introduction}\label{sec0}

\par

In the paper we characterize Pilipovi{\'c} spaces of the form
$\maclH _{\flat _\sigma}(\rr d)$ and $\maclH _{0,\flat _\sigma}
(\rr d)$, considered in \cite{FeGaTo1,Toft14},
in terms of estimates of powers of the harmonic oscillator, 
on the involved functions.

\par

The set of Pilipovi{\'c} spaces is a family of Fourier
invariant spaces, containing any Fourier invariant (standard)
Gelfand-Shilov space. The (standard) Pilipovi{\'c} spaces
$\maclH_s(\rr d)$ and $\maclH_{0,s}(\rr d)$
with respect to $s\in \mathbf R_+$, are the sets of all formal
Hermite series expansions
\begin{equation}\label{Eq:fHermite}
  f(x) = \sum _{\alpha \in \nn d}c_\alpha (f)h_\alpha (x)
\end{equation}
such that
\begin{equation}\label{Eq:Cond.}
|c_\alpha(f)| \lesssim e^{-r|\alpha |^{\frac 1{2s}}}
\end{equation}
holds true for some $r>0$ respective for every $r>0$. 
Here $f(\theta )\lesssim g(\theta )$ means that
$f(\theta )\le cg(\theta)$ for some
constant $c>0$ which is independent of $\theta$ in the
domain of $f$ and $g$. (See also \cite{Ho}
and Section 1 for notations.) Evidently, $\maclH_s(\rr d)$ and $\maclH_{0,s}(\rr d)$
increases with $s$. It is proved in \cite{Pil2} that if $\maclS_s(\rr d)$ and
$\Sigma_s(\rr d)$ are the Gelfand-Shilov spaces of Roumieu respective
Beurling type of order $s$, then
\begin{alignat}{2}
\maclH_s(\rr d) &= \maclS_s(\rr d),& \quad s &\ge \frac 12,
\label{Eq:Cond2}
\\[1ex]
\maclH_{0,s}(\rr d) &= \Sigma_s(\rr d),& \quad s &> \frac 12,
\label{Eq:Cond3}
\end{alignat}
and
$$
\maclH_{0,s}(\rr d)\neq \Sigma_s(\rr d)=\{0\},\quad s=\frac 12.
$$

\par

It is also well-known that $\maclS_s(\rr d)=\{0\}$ when $s<\frac 12$ 
and $\Sigma_s(\rr d)=\{0\}$ when $s\le\frac 12$.
These relationships are completed in \cite{Toft14} by the relations
\begin{alignat*}{2}
\maclH_s(\rr d) &\neq \maclS_s(\rr d) = \{0\},& \quad s&<\frac 12
\intertext{and}
\maclH_{0,s}(\rr d) &\neq \Sigma_s(\rr d)=\{0\},& \quad s &\le\frac 12.
\end{alignat*}
In particular, each Pilipovi\'c space is contained in the
Schwartz space $\mascS (\rr d)$.

\par

For $\maclH_s(\rr d)$ ($\maclH_{0,s}(\rr d)$) we also have the characterizations
\begin{equation}\label{Eq:CharPowerHarmOsc}
  f\in \maclH_s(\rr d)\quad (f\in \maclH_{0,s}(\rr d))\quad  \Leftrightarrow \quad
  \norm {H^{N}_{d} f}_{L^\infty}\lesssim r^N N!^{2s}
\end{equation}
for some $r>0$ (for every $r>0$) concerning estimates of
powers of the harmonic oscillator
$$
H_d=|x|^2-\Delta _x,\qquad x\in \rr d,
$$
acting on the involved functions. These
relations were obtained in \cite{Pil2}
for $s\ge \frac 12$, and in \cite{Toft14} in the general case $s>0$.

\par

In \cite{FeGaTo1,Toft14} characterizations of $\maclH_s(\rr d)$ and $\maclH_{0,s}(\rr d)$
were also obtained by certain spaces of analytic functions on $\cc d$,
via the Bargmann transform. From these mapping properties it follows that near
$s=\frac 12$ there is a jump concerning these Bargmann images. More precisely,
if $s=\frac 12$, then the Bargmann image of $\maclH _{s}(\rr d)$
(of $\maclH _{0,s}(\rr d)$) is the set of all entire functions $F$ on $\cc d$ such that
$F$ obeys the condition
\begin{equation}\label{Eq:AnalFEst1}
|F(z)|\lesssim e^{(\frac 12-r)|z|^2}
\qquad
(\, |F(z)|\lesssim e^{r|z|^2}\, )
\end{equation}
for some $r>0$ (for every $r>0$). For $s<\frac 12$,
this estimate is replaced  by
\begin{equation}\label{Eq:AnalFEst2}
|F(z)| \lesssim e^{r(\log (1+|z|))^{\frac 1{1-2s}}}
\end{equation}
for some $r>0$ (for every $r>0$), which is indeed a stronger condition compared
to the case $s=\frac 12$.

\par

An important motivation for considering the spaces 
$\maclH_{\flat_\sigma}(\rr d)$ and $\maclH_{0,\flat_\sigma}(\rr d)$ 
is to make this gap smaller. More precisely, 
$\maclH_{\flat_\sigma}(\rr d)$ and $\maclH_{0,\flat_\sigma}(\rr d)$, 
which are Pilipovi\'c spaces of Roumieu respective
Beurling type, is a family of function spaces, which increases with 
$\sigma$ and such that
$$
\maclH _{s_1}(\rr d)
\subseteq
\maclH_{0,\flat_\sigma}(\rr d)
\subseteq
\maclH_{\flat_\sigma}(\rr d)
\subseteq
\maclH _{0,s_2}(\rr d),
\qquad s_1<\frac 12,\ s_2\ge \frac 12.
$$
The spaces $\maclH_{\flat_\sigma}(\rr d)$ and
$\maclH_{0,\flat_\sigma}(\rr d)$ consist of all 
formal Hermite series expansions
\eqref{Eq:fHermite} such that
\begin{equation} \label{Eq:Cond.3}
|c_\alpha (f)|\lesssim r^{|\alpha|}\alpha!^{-\frac 1{2\sigma}}
\end{equation}
hold true for some $r>0$ respectively for every $r>0$. For the Bargmann images of
$\maclH_{\flat_\sigma}(\rr d)$ and $\maclH_{0,\flat_\sigma}(\rr d)$, the conditions
\eqref{Eq:AnalFEst1} and \eqref{Eq:AnalFEst2} above are replaced by
$$
|F(z)|\lesssim e^{r|z|^{\frac {2\sigma}{\sigma +1}}},
$$
for some $r>0$ respectively for every $r>0$. It follows that the gaps of the Bargmann
images of $\maclH _s(\rr d)$ and $\maclH _{0,s}(\rr d)$ between the cases $s<\frac 12$
and $s\ge \frac 12$ are drastically decreased by including the spaces
$\maclH_{\flat_\sigma}(\rr d)$ and $\maclH_{0,\flat_\sigma}(\rr d)$, $\sigma >0$, in the
family of Pilipovi{\'c} spaces.

\par

In \cite{FeGaTo1}, characterizations of
$\maclH_{\flat_1}(\rr d)$ and $\maclH_{0,\flat_1}(\rr d)$ in terms of
estimates of powers of the harmonic oscillator acting
on the involved functions which corresponds to
\eqref{Eq:CharPowerHarmOsc} are deduced.
On the other hand, apart from the case $\sigma =1$, it seems that no
such characterizations for $\maclH_{\flat_\sigma}(\rr d)$ and
$\maclH_{0,\flat_\sigma}(\rr d)$ have been obtained so far.

\par

In Section \ref{sec2} we fill this gap in the theory, and
deduce such characterizations. In particular, as a consequence
of our main result, Theorem \ref{Thm:Mainthm1} in Section \ref{sec2}, we have
\begin{gather*}
f\in \maclH _{\flat _\sigma }(\rr d) \quad
(f\in \maclH _{0,\flat _\sigma }(\rr d))
\\[1ex]
\Leftrightarrow
\\[1ex]
\norm {H_d^N f}_{L^\infty} \lesssim 2^N r^{\frac{N}{\log (N\sigma )}}
\left ( \frac{2N\sigma}{\log (N\sigma )} \right )^{N ( 1-\frac{1}{\log (N\sigma )} ) } 
\end{gather*}
for some (every) $r > 0$.
By choosing $\sigma =1$ we regain the corresponding characterizations in
\cite{FeGaTo1} for $\maclH_{\flat_1}(\rr d)$ and $\maclH_{0,\flat_1}(\rr d)$.

\par

\section{Preliminaries}\label{sec1}

\par

In this section we recall some facts about Gelfand-Shilov spaces,
Pilipovi{\'c} spaces and modulation spaces.

\par

Let $s>0$. Then the (Fourier invariant) Gelfand-Shilov spaces $\maclS _s(\rr d)$
and $\Sigma _s(\rr d)$ of Roumieu and Beurling type, respectively,
consist of all $f\in C^\infty (\rr d)$ such that
\begin{equation}\label{Eq:GSNorm}
\nm f{\maclS _{s;r}}\equiv
\sup _{\alpha ,\beta \in \nn d}
\left (
\frac {\nm {x^\alpha D^\beta f}{L^\infty (\rr d)}}
{r^{|\alpha +\beta|}(\alpha !\beta !)^s}
\right )
\end{equation}
is finite, for some $r>0$ respectively for every $r>0$. The
topologies of $\maclS _s(\rr d)$ and $\Sigma _s
(\rr d)$ are the inductive limit topology and the
projective limit topology, respectively,
supplied by the norms \eqref{Eq:GSNorm}.
We refer to \cite{ChuChuKim,GS}
for more facts about Gelfand-Shilov spaces.

\par

For $\maclH _s(\rr d)$ and $\maclH _{0,s}(\rr d)$ we consider the norms
\begin{alignat*}{3}
  \nm f{\maclH _{s;r}}
  &\equiv
  \sup\limits_{\alpha\in \nn d}
  (|c_\alpha(f)|e^{r|\alpha|^{\frac 1{2s}}}) &
  \quad &\text{when} & \quad s &\in \mathbf R_+
  \intertext{and}
  \nm f{\maclH _{s;r}} &\equiv \sup \limits _{\alpha\in \nn d}
  \left(
  |c_\alpha(f)| r^{-|\alpha|}\alpha!^{\frac1{2\sigma}}
  \right) &
  \quad &\text{when} & \quad s &= \flat_\sigma ,
\end{alignat*}
when $r>0$ is fixed. Then the set $\maclH _{s;r}(\rd)$ consists 
of all $f\in C^\infty (\rr d)$ such that $\nm f{\maclH _{s;r}}$
is finite. It follows that $\maclH _{s;r}(\rd)$ is a Banach 
space.

\par

The Pilipovi{\'c} spaces $\maclH _s(\rd)$ and $\maclH _{0,s}(\rd)$ 
are the inductive limit and the projective limit, respectively, of 
$\maclH _{s;r}(\rd)$ with respect to $r>0$. In particular,
$$
\maclH _s(\rd)=\bigcup_{r>0}\maclH _{s;r}(\rd)\quad \text{and}
\quad \maclH _{0,s}(\rd)=\bigcap_{r>0}\maclH _{s;r}(\rd)
$$
and it follows that $\maclH _s(\rd)$ is complete, 
and that $\maclH _{0,s}(\rd)$ is a Fr\'echet space. It is well-known 
that the identities \eqref{Eq:Cond2} and \eqref{Eq:Cond3} also hold 
in topological sense (cf. \cite{Pil2}).

\par

By extending $\mathbf R_+$ into $\mathbf R_\flat\equiv\mathbf R_+
\cup \{\flat_\sigma\}_{\sigma>0}$ and letting
$$
s_1< \flat_{\sigma _1} < \flat_{\sigma _2} <s_2 \quad \text{when}\quad s_2 \ge \frac 12,
\ s_1<\frac 12\ \text{and}\ \sigma _1<\sigma _2,
$$
we have
$$
\maclH _{s_1}(\rr d)\subseteq \maclH _{0,s_2}(\rr d)
\subseteq \maclH _{s_2}(\rr d),\quad s_1,s_2\in\mathbf R_\flat\;
\text{and}\; s_1<s_2.
$$

\medspace

We also need some facts about weights and modulation spaces, a family of
(quasi-)Banach spaces, introduced by Feichtinger in \cite{Fe1}.
A \emph{weight} on $\rr d$
is a function $\omega \in L^\infty _{loc}(\rr d)$ such that $\omega (x)>0$ for every
$x\in \rr d$ and $1/\omega \in L^\infty _{loc}(\rr d)$. The weight $\omega$ on $\rr d$
is called moderate of polynomial type, if there is an integer $N\ge 0$ such that
$$
\omega (x+y)\lesssim \omega (x)(1+|y|)^N,\qquad x,y\in \rr d .
$$
The set of moderate weights of polynomial type on $\rr d$
is denoted by $\mascP (\rr d)$.

\par

Let $p,q\in (0,\infty ]$, $\phi \in
\mascS (\rr d)\setminus 0$ and $\omega \in \mascP (\rr {2d})$ be fixed.
Then the modulation space, $M^{p,q}_{(\omega )}(\rr {d})$ consists of
all $f\in \mascS '(\rr d)$ such that
$$
\nm f{M^{p,q}_{(\omega )}} \equiv \nm {V_\phi f \cdot \omega}{L^{p,q}}
$$
is finite. Here $V_\phi f$ is the short-time Fourier transform of $f$ with respect
to $\phi$, given by
$$
V_\phi f(x,\xi ) = (2\pi )^{-\frac d2}\scal f{e^{i\scal \cdo \xi}\overline{\phi (\cdo -x)}}
$$
and
$$
\nm F{L^{p,q}} = \nm F{L^{p,q}(\rr {2d})} \equiv \nm {g_F}{L^q(\rr d)}
\quad \text{when}\quad
g_F(\xi ) = \nm {F(\cdo ,\xi )}{L^p(\rr d)}
$$
and $F$ is measurable on $\rr {2d}$.

\par

Modulation spaces possess several convenient properties,
For example we have the following proposition (see \cite{Fe1,GaSa} for proofs).

\par

\begin{prop}
Let $p,q\in (0,\infty ]$ and $\omega \in \mascP (\rr {2d})$. Then the following
is true:
\begin{enumerate}
\item $M^{p,q}_{(\omega)}(\rr d)$ is a quasi-Banach space under the quasi-norm
$\nm \cdo{M^{p,q}}$ above. If in addition $p,q\ge 1$, then $\nm \cdo{M^{p,q}}$
is a norm and $M^{p,q}_{(\omega)}(\rr d)$ is a Banach space;

\vrum

\item the definition of $M^{p,q}(\rr d)$ is independent of the choice of $\phi$
above and different $\phi \in \mascS (\rr d)\setminus 0$ gives rise to equivalent
quasi-norms;

\vrum

\item $M^{p,q}_{(\omega )}(\rr d)$ increases with $p$ and $q$
(also in topological sense).
\end{enumerate}
\end{prop}

\par

\section{Characterizations of $\maclH _{\flat _\sigma}(\rr d)$
and $\maclH _{0,\flat _\sigma} (\rr d)$ in terms of powers of
the harmonic oscillator}\label{sec2}

\par

In this section we deduce characterizations of the test function spaces
$\maclH _{0,\flat \sigma}(\rr d)$ and $\maclH _{\flat _\sigma}(\rr d)$.

\par

More precisely we have the following.

\par

\begin{thm}\label{Thm:Mainthm1}
Let $\sigma >0$, $N,N_0\in \mathbf N$ be such that $N_0\sigma >1$,
$p_0\in [1,\infty ]$, $p,q\in (0,\infty ]$,
$\omega \in \mascP (\rr {2d})$ and let $f\in C^\infty (\rr d)$ be given
by \eqref{Eq:fHermite}. Then the following conditions are equivalent:
\begin{enumerate}
\item $f\in \maclH _{\flat _\sigma}(\rr d)$ ($f\in \maclH _{0,\flat _\sigma}(\rr d)$);

\vrum

\item for some $r > 0$ (for every $r > 0$) it holds
$$
\{ c_\alpha (f)r^{-|\alpha |}(\alpha !)^{\frac{1}{2\sigma}} \} _{\alpha \in \nn d}
\in \ell ^q(\nn d)\text ;
$$

\vrum

\item for some $r > 0$ (for every $r > 0$) it holds
\begin{equation}\label{Eq:GFHarmCond}
\norm {H_d^N f}_{L^{p_0}} \lesssim 2^N r^{\frac{N}{\log (N\sigma )}}
\left ( \frac{2N\sigma}{\log (N\sigma )} \right )^{N ( 1-\frac{1}{\log (N\sigma )} ) },
\quad N  \ge N_0 \text ;
\end{equation}

\vrum

\item for some $r > 0$ (for every $r > 0$) it holds
\begin{equation}\label{Eq:GFHarmCond2}
\norm {H_d^N f}_{M^{p,q}_{(\omega )}} \lesssim 2^N r^{\frac{N}{\log (N\sigma )}}
\left ( \frac{2N\sigma}{\log (N\sigma )} \right )^{N ( 1-\frac{1}{\log (N\sigma )} ) },
\quad N  \ge N_0.
\end{equation}
\end{enumerate}
\end{thm}

\par

We need some preparations for the proof.
In the following proposition we treat separately the equivalence between (3)
and (4) in Theorem \ref{Thm:Mainthm1}.

\par

\begin{prop}\label{Prop:NormEquiv}
Let $p_0\in [1,\infty ]$, $p,q\in (0,\infty ]$, $\sigma >0$ $N_0>\sigma ^{-1}$
be an integer and let $\omega \in \mascP (\rr {2d})$. Then the following conditions
are equivalent:
\begin{enumerate}
\item \eqref{Eq:GFHarmCond} holds for some $r>0$ (for every $r>0$);

\vrum

\item  \eqref{Eq:GFHarmCond2} holds for some $r>0$ (for every $r>0$).
\end{enumerate}
\end{prop}

\par

We need the following lemma for the proof of Proposition \ref{Prop:NormEquiv}.

\par

\begin{lemma}\label{Lemma:ExpressionLogExps}
Let $R\ge e$, $I = (0,R]$,
\begin{align*}
g(r,t_1,t_2)
&\equiv
\frac
{r^{\frac {t_2}{\log t_2}}}{r^{\frac {t_1}{\log t_1}}}
\quad \text{and}\quad
h(t_1,t_2)
\equiv
\frac
{
\left (
\frac {2t_2}{\log t_2}
\right )
^{t_2(1-\frac 1{\log t_2})}
}
{
\left (
\frac {2t_1}{\log t_1}
\right )
^{t_1(1-\frac 1{\log t_1})}
},
\end{align*}
when $t_1,t_2>e$ and $r>0$.
Then
\begin{equation}\label{Eq:aAndhEst}
0 \le g(r,t_1,t_2)\le C
\quad \text{and} \quad
0\le h(t_1,t_2)\le \left (\frac{2t_1}{\log t_1} \right )^C
\end{equation}
when
$$
t_1,t_2>R,\ 0\le t_2-t_1 \le R, \ r\in I,
$$
for some constant $C>0$ which only depends on $R$.
\end{lemma}

\par

\begin{proof}
Since $t\mapsto \frac t{\log t}$ is increasing when $t\ge e$, $g$
is upper bounded by one when $r\le 1$, and the boundedness of $g$ follows in this case.

\par

If $r\ge 1$, $t=t_1$, $u =t_2-t_1>0$ and $\rho =\log r$, then
\begin{multline*}
0\le \log g(r,t_1,t_2)
=
\left (
\frac {t+u}{\log (t+u )}-\frac t{\log t}
\right )
\rho
\\[1ex]
=
\frac t{\log t}
\left (
\frac {1+\frac u t}{1+\frac {\log (1+\frac u t)}{\log t}}
-1
\right )
\rho
=
\frac t{\log t}
\left (
\frac
{\frac u t - \frac {\log (1+\frac u t)}{\log t}}
{1+\frac {\log (1+\frac u t)}{\log t}}
\right )
\rho
\\[1ex]
<
\frac t{\log t} \cdot \frac u t \cdot \rho
=
\frac {u \rho}{\log t} \le C
\end{multline*}
for some constant $C$ which only depends on $r$ and $R$. This shows
the boundedness of $g$.

\par

Next we show the estimates for $h(t_1,t_2)$ in \eqref{Eq:aAndhEst}. By taking the
logarithm of $h(t_1,t_2)=h(t,t_2)$ we get
$$
\log h(t,t_2)=t_2
\log
\left (
  \frac{2t_2}{\log t_2}
  \right )
  -t \log
  \left (
  \frac{2t}{\log t}
  \right )
- b(t,t_2),
$$
where
$$
b(t,t_2)
=
\left (
\frac{t_2}{\log t_2}
\log
\left (
\frac{2t_2}{\log t_2}
\right )
-\frac t{\log t} \log
\left (
\frac{2t}{\log t}
\right )
\right ).
$$
Since $b(t,t_2)>0$ when $t_2>t$, we get
\begin{multline*}
\log h(t_1,t_2)<
t_2
\log
\left (
  \frac{2t_2}{\log t_2}
\right )
-t \log
\left (
\frac{2t}{\log t}
\right )
\\[1ex]
=(t+u )
\left (
\log
\left (
\frac{2t}{\log t}
\right )
+\log\left (
\frac{1+\frac{u}{t}}
{1+\frac{\log (1+\frac{u}{t})}{\log t}}
\right )
\right )
-t\log
\left (
\frac{2t}{\log t}
\right )
\\[1ex]
\le
u \log
\left (
\frac{2t}{\log t}
\right )
+t \log
\left (
1+\frac{u}{t}
\right )
+C
\\[1ex]
\le
u \log
\left (
\frac{2t}{\log t}
\right )
+u +C
\end{multline*}
for some constant $C\ge 0$. Here we have used that $t_1,t_2>R\ge e$
and the fact that $t\mapsto \frac t{\log t}$ increases for $t\ge R$.
\end{proof}

\par

\begin{proof}[Proof of Proposition \ref{Prop:NormEquiv}]
First we prove that \eqref{Eq:GFHarmCond2} is independent of $N_0>\sigma ^{-1}$
when $p,q\ge 1$. Evidently, if \eqref{Eq:GFHarmCond2} is true for $N_0$,
then it is true for any larger replacement of $N_0$. On the other hand, the map
\begin{equation}\label{HarmonicOscModMap}
H_d^N \, :\, M^{p,q}_{(v_N\omega )}(\rr d)\to M^{p,q}_{(\omega )}(\rr d),
\qquad v_N(x,\xi )=(1+|x|^2+|\xi |^2)^N,
\end{equation}
and its inverse are continuous and bijective (cf. e.{\,}g.
\cite[Theorem 3.10]{SiTo}). Hence, if $\sigma ^{-1}< N_1\le N_0$,
$N_2=N_0-N_1\ge 0$ and \eqref{Eq:GFHarmCond2} holds for
$N_0$, then
$$
\nm {H_d^{N_1}f}{M^{p,q}_{(\omega )}}
\lesssim
\nm {H_d^{N_0}f}
{M^{p,q}_{(\omega /v_{N_2})}}
\lesssim
\nm {H_d^{N_0}f}{M^{p,q}_{(\omega )}}<\infty ,
$$
and \eqref{Eq:GFHarmCond2} holds for $N_1$. This implies that
\eqref{Eq:GFHarmCond2} is independent of $N_0>\sigma ^{-1}$ when $p,q\ge 1$.

\par

Next we prove that \eqref{Eq:GFHarmCond2} is independent of the
choice of $\omega \in \mascP (\rr {2d})$. By the first part of the proof, we may
assume that $N_0\sigma >e$.
For every $\omega _1,\omega _2\in \mascP (\rr {2d})$, we may find an integer
$N_0\ge 0$ such that
$$
\frac 1{v_{N_0}}\lesssim \omega _1, \omega _2\lesssim v_{N_0},
$$
and then
\begin{equation}\label{NormArrayCond}
\nm f{M^{p,q}_{(1/v_{N_0})}}\lesssim \nm f{M^{p,q}_{(\omega _1)}},
\nm f{M^{p,q}_{(\omega _2)}}
\lesssim \nm f{M^{p,q}_{(v_{N_0})}}.
\end{equation}
Hence the stated invariance follows if we prove that
\eqref{Eq:GFHarmCond2} holds for $\omega =v_{N_0}$, if it is true for
$\omega =1/v_{N_0}$.

\par

Therefore, assume that \eqref{Eq:GFHarmCond2} holds for $\omega =1/v_{N_0}$.
Let $f_N=H_d^{N}f$, $u =2N_0\sigma$, $t=t_1=N\sigma$, $N_2=N+2N_0$ and
$t_2=t_1+u = N_2\sigma$.
If $N\ge 2N_0$, then the bijectivity of \eqref{HarmonicOscModMap} gives
\begin{multline} \label{Eq:Comp}
  \frac{\nm {f_N}{M^{p,q}_{(v_{N_0})}}^\sigma}
  {2^{N\sigma}
  r^{\frac {N\sigma}{\log (N\sigma)}}
\left (
\frac {2N\sigma}{\log (N\sigma)}
\right )
^{N\sigma(1-\frac 1{\log (N\sigma )})}}
=
\frac{\nm {f_N}{M^{p,q}_{(v_{N_0})}}^\sigma}
{{2^t
  r^{\frac t{\log t}}
\left (
\frac {2t}{\log t}
\right )
^{t(1-\frac 1{\log t})}}}
\\[1ex]
\lesssim
\frac{\nm {f_{N+2N_0}}{M^{p,q}_{(1/v_{N_0})}}^\sigma}
{{2^t
  r^{\frac t{\log t}}
\left (
\frac {2t}{\log t}
\right )
^{t(1-\frac 1{\log t})}}}
\\[1ex]
=
2^{u}
g(r,t_1,t_2)h(t_1,t_2)
\cdot
\frac{\nm {f_{N_2}}{M^{p,q}_{(1/v_{N_0})}}^\sigma}
{2^{N_2\sigma}r^{\frac {t_2}{\log (t_2)}}
\left (
\frac {2t_2}{\log t_2}
\right )
^{t_2(1-\frac 1{\log t_2})}},
\end{multline}
where $g(r,t_1,t_2)$ and $h(t_1,t_2)$ are the same as in Lemma
\ref{Lemma:ExpressionLogExps}. A combination of Lemma
\ref{Lemma:ExpressionLogExps}, \eqref{Eq:Comp} and the fact that
$N\sigma >e$ shows that
(2) is independent of $\omega \in \mascP (\rr {2d})$.
For general
$p,q>0$, the invariance of \eqref{Eq:GFHarmCond2} with respect to $\omega$,
$p$ and $q$ is a consequence of the embeddings
$$
M^\infty _{(v_N\omega )}(\rr d)\subseteq M^{p,q} _{(\omega )}(\rr d)
\subseteq M^\infty _{(\omega )}(\rr d),
\qquad
N>
d\left (
\frac 1p +\frac 1q
\right )
$$
(see e.{\,}g. \cite[Theorem 3.4]{GaSa} or \cite[Proposition 3.5]{Toft13}).

\par

The equivalence between (1) and (2) now follows from
these invariance properties and the continuous
embeddings
$$
M^{p_0,q_1}\subseteq L^{p_0}\subseteq M^{p_0,q_2},
\qquad
q_1=\min (p_0,p_0'),\quad q_2=\max (p_0,p_0'),
$$
which can be found in e.{\,}g. \cite[Proposition 1.7]{Toft8}.
\end{proof}

\par

\begin{prop}\label{Prop:CoeffGivesHarmPowEst}
Let $f\in C^\infty (\rr d)$ and $\sigma >0$. If
\begin{equation}\label{Eq:CoeffGivesHarmPowEst1}
\nm {H_d^Nf}{L^2}\lesssim 2^Nr^{\frac N{\log (N\sigma)}}
\left (
\frac {2N\sigma}{\log (N\sigma)}
\right )
^{N(1-\frac 1{\log (N\sigma )})},\quad
N\in \mathbf N,\ N\sigma \ge e,
\end{equation}
for some $r>0$ (for every $r>0$), then
\begin{equation}\label{Eq:CoeffGivesHarmPowEst2}
|c_\alpha (f)|
\lesssim
r^{|\alpha|}|\alpha |^{-\frac {|\alpha |}{2\sigma}},
\quad
\alpha \in \nn d,
\end{equation}
for some $r>0$ (for every $r>0$).
\end{prop}

\par

\begin{prop}\label{Prop:HarmPowGivesCoeffEst}
Let $f\in C^\infty (\rr d)$ and $\sigma >0$. If
\eqref{Eq:CoeffGivesHarmPowEst2} holds for some $r>0$
(for every $r>0$), then \eqref{Eq:CoeffGivesHarmPowEst1}
holds for some $r>0$ (for every $r>0$).
\end{prop}

\par

For the proofs we need some preparation lemmas.

\par

\begin{lemma}\label{Lemma:LogFuncEst}
Let $\sigma>0$, $\sigma _0 \in [0,\sigma ]$ and let
$$
F(r,t)=\left(
\frac{2t}{\log t}
\right )
^{t( 1-\frac{1}{\log t} )}
r^{\frac{t}{\log t}},
\qquad r\geq 0, \quad t\ge e\cdot \max (1,\sigma ) .
$$
Then
\begin{alignat}{2}
F(r,t) &\leq F(r,t+\sigma _0),& \quad r &\in [1,\infty ),
\label{Eq:LogFuncEst1}
\intertext{and}
F(r,t) &\leq F(r^{1-\frac 1{e}},t+\sigma _0),& \quad r&\in (0,1].
\label{Eq:LogFuncEst2}
\end{alignat}
\end{lemma}

\par

\begin{proof}
If $r\ge 1$, then it follows by straight-forward tests with
derivatives that $F(r,t)$ is increasing with respect to $t\ge e$.
This gives \eqref{Eq:LogFuncEst1}.

\par

In order to prove \eqref{Eq:LogFuncEst2}, let
$t_1=t+\sigma _0$ and
$$
h(t_1,\sigma _0 )
=
\frac{1-\frac{\sigma _0}{t_1}}
{1+\frac{\log (1-\frac{\sigma _0}{t_1})}{\log t_1}},
$$
where $0\leq \sigma _0 \leq \sigma$. Then
\begin{equation}\label{Eq:LogFractionEst}
\left(
\frac{2t}{\log t}
\right)^{t(1-\frac{1}{\log t})}
r^{\frac{t}{\log t}}
\leq
\left(
\frac{2t_1}{\log t_1}
\right)^{t_1(1-\frac{1}{\log t_1})}
r^{\frac{t}{\log t}}
\end{equation}
and
\begin{equation*}
{\frac{2t}{\log t}}=
h(t_1,\sigma _0 )
\cdot \frac{2t_1}{\log t_1}.
\end{equation*}
Since
$$
0\le \frac {\sigma _0}{t_1}\le \frac 1e
\quad \text{and}\quad
-1<\frac {\log (1-\frac {\sigma _0}{t_1})}{\log t_1}\le 0
$$
we get
$$
h(t_1,\sigma _0 ) \ge 1-\frac {\sigma _0}{t_1}\ge 1-\frac 1e.
$$
Hence the facts $\frac {t_1}{\log t_1} \ge 1$ and $0<r\le 1$
give
$$
r^{\frac t{\log t}}= r^{h(t_1,\sigma _0 )\frac {t_1}{\log t_1}}
\le
r^{(1-\frac 1e)\frac {t_1}{\log t_1}}.
$$

\par

A combination of the latter inequality with
\eqref{Eq:LogFractionEst} gives
$$
F(r,t)
\le
\left(
\frac{2t_1}{\log t_1}
\right)^{t_1(1-\frac{1}{\log t_1})}
\left (
r^{(1-\frac 1e)}
\right )
^{\frac{t_1}{\log t_1}}
= F(r^{1-\frac 1e},t_1).
\qedhere
$$
\end{proof}

\par

\begin{lemma}\label{Lemma:CoeffGivesHarmPowEst}
Let $s\ge \sigma (e+1)+e^2$ 
$$
\Omega _1
=
[e,\infty )\cap (\sigma \cdot \mathbf N)
\quad \text{and}\quad
\Omega _2
=
[e,\infty ).
$$
Then the following is true:
\begin{enumerate}
\item for any $r_2>0$, there is an $r_1>0$ such that
\begin{equation}\label{Eq:GeneralEstimate1}
\inf _{t\in \Omega _j}
\left (
s^{-t}
\left (
\frac {2t}{\log t}
\right )
^{t(1-\frac 1{\log t})}r_1^{\frac t{\log t}}
\right )
\lesssim
r_2 ^ss^{-\frac s2},\quad j=1,2\text ;
\end{equation}

\vrum

\item for any $r_1>0$, there is an $r_2>0$ such that
\eqref{Eq:GeneralEstimate1} holds.
\end{enumerate}
\end{lemma}

\par

\begin{proof}
First prove the result for $j=2$. Let
$$
x=\log t,\quad y=\log s \ge \log (\sigma (e+1)+e^2) >2
\quad
\rho _j=\log r_j,\quad j=1,2.
$$
By applying the logarithm on \eqref{Eq:GeneralEstimate1},
the statements (1) and (2) follow if we prove:
\begin{enumerate}
\item[(1)$'$] for any $\rho _2\in \mathbf R$, there is a
$\rho _1\in \mathbf R$ such that
\begin{equation}\label{Eq:GeneralEstimate2}
\inf _{x\ge x_0}
F(x) \le 0,
\qquad x_0
=
\log (\sigma (e+1)+e^2)
\end{equation}
where
\begin{equation}\label{Eq:FDef}
F(x)=
-e^xy + e^x
\left (
1-\frac 1x
\right )
(x+\log 2-\log x)
+
\rho _1 \frac {e^x}x-\rho _2e^y+\frac {e^yy}2
\end{equation}

\vrum

\item[(2)$'$] for any $\rho _1\in \mathbf R$, there
is a $\rho _2\in \mathbf R$ such that
\eqref{Eq:GeneralEstimate2} holds.
\end{enumerate}

\par

We choose
$$
x=y+\log y -\log 2 \geq \log s \geq x_0
\quad \text{and let}\quad
h=g(y),
$$
where
$$
g(u)=\frac {\log u - \log 2}u.
$$
Obviously, $x$ increases with $y$, and by function investigations
it follows that
$$
0=g(2)<g(u)\le g(2e)=\frac 1{2e},\qquad u>2,
$$
giving that $0<h\le \frac 1{2e}<1$.
Then \eqref{Eq:FDef} becomes
\begin{multline*}
e^{-y}F(y+\log y-\log 2)
\\[1ex]
=
-\frac {y^2}2 +
\frac y2
\left (
1-\frac 1{y+\log \frac y2}
\right )
(y+\log y -\log (y+\log \frac y2))
+\frac {\rho _1y}{2(y+\log \frac y2)}-\rho _2+\frac y2
\\[1ex]
=
-\frac y2\log (1+h)+\frac {\log y +\log (1+h)}{2(1+h)}
+\frac {\rho _1-\log 2}{2(1+h)}-\rho _2 .
\end{multline*}
If $\rho _1\in \mathbf R$ is fixed, then we choose
$\rho _2\in \mathbf R$ such that
\begin{equation}\label{Eq:ChoosingRho}
\frac {\rho _1-\log 2}{2(1+h)}-\rho _2\le -C_0
\end{equation}
for some large number $C_0>0$.
In the same way, if $\rho _2\in \mathbf R$ is fixed,
then we choose $\rho _1\in \mathbf R$ such that
\eqref{Eq:ChoosingRho} holds. For such choices
and the fact that $0<h<1$, the inequality
$$
0<h-\frac {h^2}2 \le \log (1+h)\le h
$$
gives
\begin{multline*}
F(y+\log y-\log 2)
\le
e^y
\left (
-\frac y2\log (1+h)+\frac {\log y +\log (1+h)}{2(1+h)}-C_0
\right )
\\[1ex]
\le
e^y
\left (
-\frac y2\log (1+h)+\frac {\log y +\log (1+h)}{2}
-C_0\right )
\\[1ex]
\le
e^y
\left (
-\frac {\log y-\log 2}2+\frac {(\log y-\log 2)^2}{4y}+\frac 12
\left (
\log y +h
\right )
-C_0\right )
\\[1ex]
\le
e^y
\left (
\frac 12\log 2+\frac {(\log y-\log 2)^2}{4y}+\frac h2
-C_0\right )<0,
\end{multline*}
provided $C_0$ was chosen large enough. This gives the result in
the case $j=2$.

\par

Next we prove the result for $j=1$. Let $r_2>0$. By the
first part of the proof, there are $t_1\ge e(\sigma +1) + \sigma$
and $r_0>0$ such that
$$
s^{-t_1}
\left (
\frac {2t_1}{\log t_1}
\right )
^{t_1(1-\frac 1{\log t_1})}r_0^{\frac {t_1}{\log t_1}}
\leq
r_2^ss^{-\frac s2}.
$$
Let $r_1= r_0$ if $r_0\geq 1$ and $r_1=r_0^{\frac{e}{e-1}}$ otherwise.
By Lemma \ref{Lemma:LogFuncEst} it follows that
$$
s^{-t}
\left (
\frac {2t}{\log t}
\right )
^{t(1-\frac 1{\log t})}r_1^{\frac {t}{\log t}}
\leq
r_2^s s^{-\frac s2}
$$
holds when $t=N\sigma$ and $N\in \mathbf N$ is chosen such that
$0\le t_1-N\sigma\le \sigma$. Observe that Lemma \ref{Lemma:LogFuncEst}
can be applied since $N\sigma > e(\sigma +1)$. This gives (1) for $j=1$.

\par

By similar arguments, (2) for $j=1$ follows from
(2) in the case $j=2$. The details are left for the reader.
\end{proof}

\par

\begin{proof}[Proof of Proposition
\ref{Prop:CoeffGivesHarmPowEst}]
Suppose that \eqref{Eq:CoeffGivesHarmPowEst1}
holds for some $r=r_1>0$. By
\begin{equation} \label{c-alpha}
c_\alpha (H_d^Nf) = (2|\alpha |+d)^Nc_\alpha (f),
\quad
|c_\alpha (H_d^Nf)|\le \nm {H_d^Nf}{L^2}
\end{equation}
and \eqref{Eq:CoeffGivesHarmPowEst1} we get
\begin{multline*}
|c_\alpha (f)|
=
\frac {|c_\alpha (H_d^Nf)|}{(2|\alpha |+d)^N}
\\[1ex]
\lesssim
\left (
|\alpha |+\frac d2
\right )^{-N}
r_1^{\frac N{\log (N\sigma )}}
\left (
\frac {2N\sigma}{\log (N\sigma )}
\right )
^{N(1-\frac 1{\log (N\sigma )})}
\\[1ex]
\le
\left (
|\alpha | ^{-N\sigma}
r_1^{\frac{N\sigma}{\log (N\sigma )}}
\left (
\frac {2N\sigma}{\log (N\sigma )}
\right )
^{N\sigma (1-\frac 1{\log (N\sigma )})}
\right )
^{\frac 1\sigma}.
\end{multline*}
By taking the infimum over all $N\ge 0$, it follows from
Lemma \ref{Lemma:CoeffGivesHarmPowEst} (2) that
\begin{equation*}
|c_\alpha (f)|
\lesssim
\left (
r_2^{|\alpha |}|\alpha |^{-\frac {|\alpha |}2}
\right )
^{\frac 1\sigma}
=
r^{|\alpha |}|\alpha |^{-\frac {|\alpha |}{2\sigma}},
\qquad |\alpha |\ge 2\sigma (e+1)+e^2,
\end{equation*}
for some $r_2>0$, where $r=r_2^{\frac 1\sigma}$.
Hence \eqref{Eq:CoeffGivesHarmPowEst2} holds for some
$r>0$.

\par

By similar arguments, using (1) instead of (2) in Lemma
\ref{Lemma:CoeffGivesHarmPowEst}, it follows that if
\eqref{Eq:CoeffGivesHarmPowEst1} holds for every $r>0$,
then \eqref{Eq:CoeffGivesHarmPowEst2} holds for every
$r>0$.
\end{proof}

\par

For the proof of Proposition \ref{Prop:HarmPowGivesCoeffEst}
we will use the following result which is essentially a slight clarification of
\cite[Lemma 2]{FeGaTo1}. The proof is therefore omitted.

\par

\begin{lemma}\label{Lemma:EstOnExFracFunc}
Let $r>0$ and
$$
f(s,t,r) = \frac {s^{2t}(2re)^s}{s^s},\qquad s> 1,\ t\ge 0.
$$
Then there exist a positive
increasing function $\theta$ on $[0,\infty )$ and a
constant $t_0=t_0(r)>e$ which only depends on $r$ such that
\begin{equation}\label{Eq:EstOnExFracFunc1}
\max _{s>0}f(s,t,r)
\le
\left (
\frac {2t}{\log t}
\right )
^{2t(1-\frac 1{\log t})}(\theta (r)r)^{\frac {2t}{\log t}},
\quad t\ge t_0(r)
\text .
\end{equation}
\end{lemma}

\par

\begin{rem}\label{Rem:EstOnExFracFunc}
The constants $s$, $t$ and $t_0(r)$ in Lemma \ref{Lemma:EstOnExFracFunc}
are denoted by $t$, $N$ and $N_0(r)$, respectively in Lemmas 1 and 2 in
\cite{FeGaTo1}. In the latter results it is understood that $N$ and $N_0(r)$
are integers. On the other hand, it is evident from the proofs of these results
that they also hold when $N$ and $N_0(r)$ are allowed to be in $\mathbf R_+$.
\end{rem}

\par

\begin{proof}[Proof of Proposition
\ref{Prop:HarmPowGivesCoeffEst}]
Let $\theta$ be as in Lemma \ref{Lemma:EstOnExFracFunc}
and $\rho \in (0,1)$. Suppose that \eqref{Eq:CoeffGivesHarmPowEst2}
holds for some $r>0$ and let $r_2>r$. From
\eqref{Eq:CoeffGivesHarmPowEst2} and \eqref{c-alpha}
we get
\begin{multline*}
\nm {H_d^Nf}{L^2}^2
=
\sum \limits _{\alpha \in \nn d}
|(2|\alpha|+d)^Nc_\alpha(f)|^2
\\[1ex]
\lesssim
\sup _{ |\alpha |\ge 1}
\left(
(2|\alpha|+d)^{2N} r_2^{2|\alpha|}
|\alpha|^{-\frac{|\alpha|}{\sigma}}
\right)
\\[1ex]
=\sup \limits _{s\ge 1}
\left(
2^{2t}
\left(
s+\frac d{2}
\right)^{2t}
r_2^{2s} s^{-s}
\right)^{\frac 1{\sigma}},
\end{multline*}
where $s=|\alpha|$ and $t=N\sigma$.
Since $0<\rho <1$ we have
\begin{multline*}
s^s=
\left(
s-\frac d{2}+\frac d{2}
\right)^{s-\frac d{2}}
s^{\frac d{2}}
=
\left(
s-\frac d{2}
\right)^{s-\frac d{2}}
s^{\frac d{2}}
\left(
1+\frac d{2s-d}
\right)^{s-\frac d{2}}
\\[1ex]
\le
\left(
s-\frac d{2}
\right)^{s-\frac d{2}}
(se)^{\frac d{2}}
\lesssim
\left(
s-\frac d{2}
\right)^{s-\frac d{2}}
\rho^{-2s}.
\end{multline*}

\par

This gives
\begin{multline} \label{Eq:equ2}
  \nm {H_d^Nf}{L^2}^2
  \lesssim
  \sup \limits_{s\ge 1}
  \left(
  2^{2t}
  \left(
  s+\frac d{2}
  \right)^{2t}
  r_2^{2s} s^{-s}
  \right)^{\frac 1{\sigma}}
  \\[1ex]
  =
  \sup \limits_{s\ge 1+\frac d{2}}
  \left(
  2^{2t} s^{2t} r_2^{2s-d}
  \left(
  s-\frac d{2}
  \right)^{-(s- \frac d{2})}
  \right)^{\frac 1{\sigma}}
  \\[1ex]
  \lesssim
  \sup \limits _{s\ge 1+\frac d{2}}
  \left (
   2^{2t} s^{2t}
   \left(
   \frac{r_2}{\rho}
   \right)^{2s}
   s^{-s}
   \right)^{\frac 1{\sigma}}.
\end{multline}

\par

Using \eqref{Eq:equ2} and Lemma
\ref{Lemma:EstOnExFracFunc} we obtain
\begin{multline*}
\nm {H_d^Nf}{L^2}^2
\lesssim
\sup \limits _{s\ge 1+\frac d{2}}
  \left(
   2^{2t} s^{2t}
   \left(
   \frac{r_2}{\rho}
   \right)^{2s}
   s^{-s}
   \right)^{\frac 1\sigma}
\\[1ex]
=
\sup \limits _{s\ge 1+\frac d{2}}
  \left(
   2^{2t} s^{2t}
   \left(
   2r_3e
   \right)^s
   s^{-s}
   \right)^{\frac 1\sigma}
\\[1ex]
\lesssim
\left (
2^{2t}
\left (
\frac{2t}{\log t}
\right )^{2t(1-\frac 1{\log t})}
\left (
\theta
(r_3)r_3
\right)^{\frac{2t}{\log t}}
\right)^{\frac 1\sigma}
\\[1ex]
=
2^{2N}(r_3\theta (r_3))^{\frac{2N}{\log (N\sigma)}}
\left (
\frac {2N\sigma}{\log (N\sigma)}
\right )
^{2N(1-\frac 1{\log (N\sigma )})}
\end{multline*}
when
$$
r_3 = \frac{r_2^{2}}{2\rho^2e}
\quad \text{and}\quad
N\sigma \ge t_0(r_3).
$$
This gives the result in the Roumieu case.

\par

By similar argument, using the fact that the non-negative
function $\theta$ is increasing, it also follows that
\eqref{Eq:CoeffGivesHarmPowEst1} holds for every $r>0$
when \eqref{Eq:CoeffGivesHarmPowEst2}
holds for every $r>0$, and the result follows.
\end{proof}

\par

\begin{proof}[Proof of Theorem \ref{Thm:Mainthm1}]
We have
\begin{multline*}
\nm {\{ c_\alpha (f)r^{-|\alpha |}(\alpha !)^{\frac{1}{2\sigma}} \}
_{\alpha \in \nn d}}{\ell ^\infty (\nn d)}
\le
\nm {\{ c_\alpha (f)r^{-|\alpha |}(\alpha !)^{\frac{1}{2\sigma}} \}
_{\alpha \in \nn d}}{\ell ^q (\nn d)}
\\[1ex]
\lesssim
\nm {\{ c_\alpha (f)(cr)^{-|\alpha |}(\alpha !)^{\frac{1}{2\sigma}} \}
_{\alpha \in \nn d}}{\ell ^\infty (\nn d)}
\end{multline*}
when $c\in (0,1)$, which shows that (2) is independent of the choice of
$q$. The equivalence between (1) and (2) now follows by the definitions
and choosing $q=\infty$ in (2).

\par

By Proposition \ref{Prop:NormEquiv} we may assume that
$p=2$. The result now follows from Propositions
\ref{Prop:CoeffGivesHarmPowEst} and \ref{Prop:HarmPowGivesCoeffEst},
together with the fact that
$$
(d\cdot e)^{-|\alpha|} |\alpha|^{|\alpha|}
\le \alpha !
\le |\alpha|^{|\alpha|},\quad \alpha \in \nn d. \qedhere
$$
\end{proof}

\par

\end{document}